\documentclass[letterpaper, 10 pt, conference]{ieeeconf}  

\IEEEoverridecommandlockouts                              

\usepackage{amsmath} 
\usepackage{amssymb}  

\usepackage{graphicx}
\usepackage{xcolor}
\usepackage[ruled,vlined]{algorithm2e}

\usepackage{balance}

\def\cM{{\mathcal{M}}}
\def\cN{{\mathcal{N}}}
\def\cP{{\mathcal{P}}}

\def\gaussmean{{a}}

\def\RR{{\mathbb{R}}}
\newcommand{\norm}[1]{\left\lVert#1\right\rVert}
\def\minwrt[#1]{\underset{#1}{\mathrm{minimize }}}
\def\argminwrt[#1]{\underset{#1}{\text{arg min }}}

\newtheorem{example}{Example}
\newtheorem{prop}{Proposition}

\title{\LARGE \bf
Mixtures of ensembles:\\ System separation and identification via optimal transport\\
}

\author{Filip Elvander$^{1}$ and Isabel Haasler$^{2}$
\thanks{$^{1}$Filip Elvander is with the Department of Information~and Communications Engineering, Aalto University, Finland.
        {\tt\small filip.elvander@aalto.fi}}%
\thanks{$^{2}$Isabel Haasler is with the Division of Systems and Control, Department of Information Technology, Uppsala University, Sweden.
        {\tt\small isabel.haasler@it.uu.se}}%
}

\begin{document}

\maketitle
\thispagestyle{empty}
\pagestyle{empty}

\begin{abstract}

Crowd dynamics and many large biological systems can be described as populations of agents or particles, which can only be observed on aggregate population level.
Identifying the dynamics of agents is crucial for understanding these large systems.
However, the population of agents is typically not homogeneous, and thus the aggregate observations consist of the superposition of multiple ensembles each governed by individual dynamics.
In this work, we propose an optimal transport framework to
jointly separate the population into several ensembles and
identify each ensemble’s dynamical system, based on aggregate observations of the population. 
We propose a bi-convex optimization problem, which we solve using a block coordinate descent with convergence guarantees.
In numerical experiments, we demonstrate that the proposed approach exhibits close-to-oracle performance also in noisy settings, yielding accurate estimates of both the ensembles and the parameters governing their dynamics.

\end{abstract}

\section{Introduction}

Dynamical systems describe many phenomena of large populations of agents, or particles, in control theory \cite{dogbe2010modeling}, biology \cite{zhao2003dynamical}, and economics \cite{brock2018nonlinearity}.
In reality, the dynamical models of these systems are often not known, and identifying them helps to uncover the fundamental principles governing the physical phenomena in order to make predictions and develop interventions.
For example, understanding crowd dynamics helps making traffic predictions for designing efficient and safe evacuation strategies \cite{kachroo2008pedestrian}, whereas identifying disease dynamics aids in developing treatment plans for patients \cite{wolkenhauer2013road}.

Often, only aggregate observations of the agents or particles are available. This may be due to, e.g., privacy concerns, economic constraints, or physiological reasons.
For example, human mobility data can typically not be collected on individual level, and medical data is usually anonymized; large flocks of birds or swarms of fish are more efficiently tracked by following the population rather than individual agents;
measuring single-cell data destroys the cell in the process, and thus dynamics can only be modeled on population level \cite{bunne2023learning, schiebinger2019optimal}.
However, most populations are heterogeneous and consist of several subpopulations, or ensembles, that are governed by different dynamics. 
For example, cell dynamics are known to be heterogeneous \cite{goldman2019impact, spiller2010measurement}, and thus an important problem in medical research is to identify their subpopulations \cite{buettner2015computational, sun2022identifying}.
Furthermore, crowd dynamics are typically heterogeneous \cite{szwaykowska2015collective}.
For example, in human crowd dynamics, groups of agents can have different objectives and movement patterns \cite{scherrer2018travelers, wu2019inferring}, whereas in animal swarms, the population can even consist of several different species \cite{ward2018cohesion}.
In these cases, aggregate observations result from the superposition of multiple ensembles, each following its own dynamical model.

Motivated by this, we in this work address the challenge of jointly separating a (heterogeneous) population into underlying ensembles and identifying each ensemble's dynamical system, based on aggregate observation of the full population.
Solving this problem has broad implications.
For example, in the analysis of single-cell data, ensemble-level information can improve the prediction of an individual's disease stage and treatment response \cite{sun2022identifying}. 
In crowd dynamics, knowledge about different ensembles in the crowd and their underlying objectives and movement patterns could be used to improve traffic planning. 
Despite its importance, it is inherently difficult to disentangle distinct dynamical ensembles from aggregated data, as exemplified in Figure \ref{fig:ex_3groups_illustration}.
\begin{figure}
    \centering
    \includegraphics[trim={40 0 40 0}, clip, width=\linewidth]{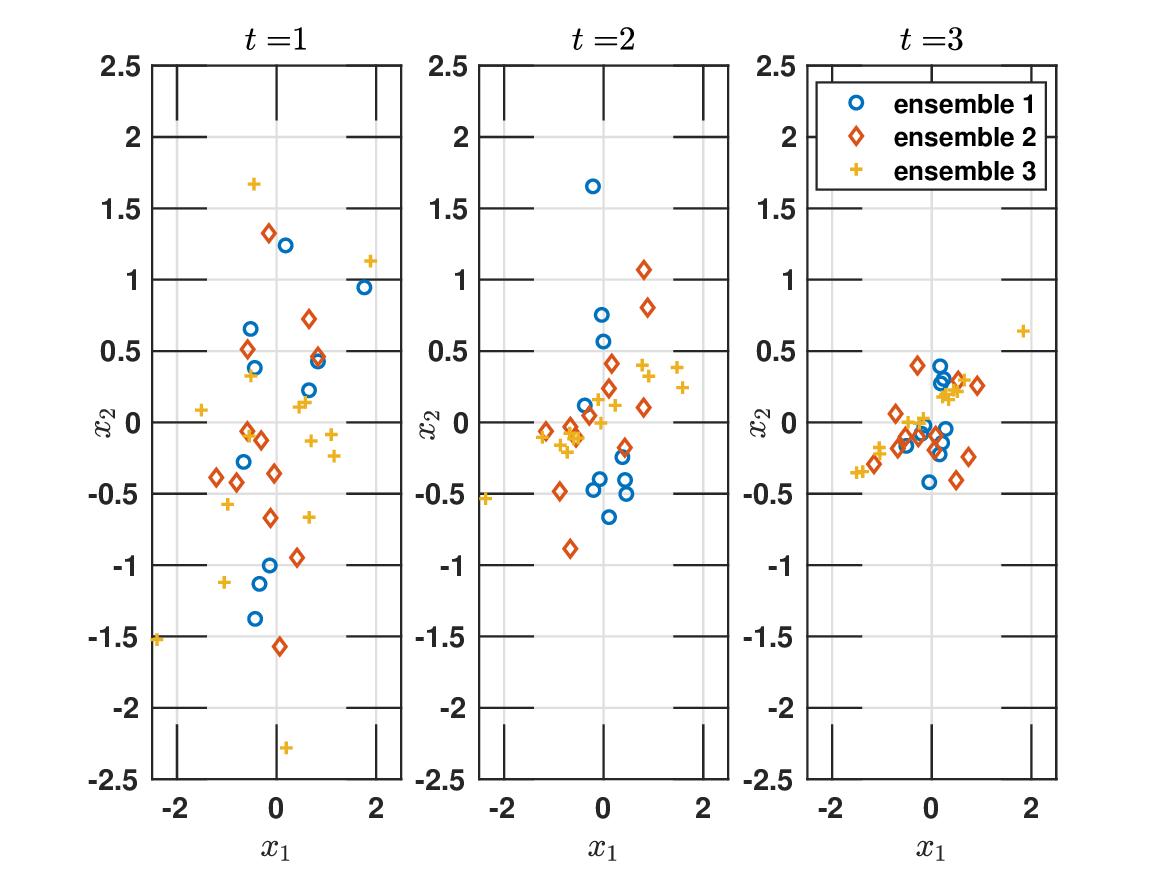}
    \caption{Time evolution of systems belonging to $K = 3$ different ensembles with different linear dynamics for the state $x^{(t)} = (x_1,x_2) \in \RR^2$. Three consecutive time steps $t$ are shown.}
    \label{fig:ex_3groups_illustration}
\end{figure}
Here, we have modeled three ensembles of agents with state space in $\RR^2$ and different linear dynamics. From aggregate snapshot observations of the entire population, it is impossible to distinguish the groups by eye.
The method we propose accurately assigns particles to ensembles and identifies each ensemble's dynamics.
Although this is a prevalent problem in a multitude of applications, we are not aware of any previous approach addressing it.

Herein, we propose a framework for performing joint ensemble separation and dynamical system identification based on the concept of optimal transport.
Optimal transport theory provides estimation and control principles for distributions \cite{chen2021optimal, haasler2021control}, and has previously been used to model the collective behavior of large multi-agent systems \cite{chen2016optimal, singh2022inference}.
Some recent works have augmented these methods to consider multiple classes of agents that can have different dynamics, for instance multi-commodity network flows and multi-species mean field games \cite{benamou2019entropy, haasler2020optimal, haasler2024scalable, ringh2023mean, ringh2024graph}.
Note that in these works the dynamics and the class assignments of the agents are known, whereas in the present paper we consider the inverse problem of identifying them.

Optimal transport has also been used for modeling the evolution of cell populations \cite{schiebinger2019optimal, bunne2023learning}.
In this context, recent work has addressed the problem of identifying a linear system that describes the observed data, assuming a homogeneous population \cite{lamoline2024gene}. This can be seen as a type of inverse optimal transport problem \cite{stuart2020inverse} in which the parameters of the linear system are optimized jointly with the transport problem itself.
Our work can be seen as a generalization of this work to the setting of heterogeneous populations where, on top of identifying the dynamics, we also have to identify the assignments to the different classes.

To the best of our knowledge, previous works on joint system identification of several systems consider only settings, where individual particle trajectories are available \cite{chen2022learning, fernex2021cluster, hsu2020linear}.
In these settings, the tasks of identifying the dynamics and clustering them can be decoupled.
In contrast to this, we consider the setting where only aggregate observations on population level are available.
Our optimal transport framework assumes indistinguishable agents and solves their association between several time instances implicitly.
It may be noted that there has recently been work on combining OT-based clustering and parameter estimation in the context of maximum likelihood estimation for Gaussian mixture models \cite{diebold2024unified,rigollet2018entropic, vayer2025note}, which although superficially similar, is distinct from the framework proposed herein.

The paper is structured as follows.
In Section \ref{sec:prel} we introduce the optimal transport problem and motivate our novel methodology on a simple illustrative example.
In Section \ref{sec:form} we state the problem formally and present an algorithm with convergence guarantees.
Finally, we demonstrate the method's performance in numerical experiments in Section~\ref{sex:exp}.
\section{Preliminaries} \label{sec:prel}
In this Section we introduce the optimal transport problem and demonstrate the intuition for our new framework for system separation and identification on a simple toy example.

\subsection{Optimal transport}
Consider two non-negative measures $\mu,\nu \in \cM_+(X)$ over a space $X$ with the same total mass.
The optimal transport problem is to find the most efficient way to move the mass from $\mu$ to $\nu$ with respect to an underlying cost function $c:X\times X \to \RR_+$, where $c(x,y)$ describes the cost for moving a unit mass from $x\in X$ to $y \in X$.
The optimal transport between the given measures is described by a transport plan, which is a measure $m \in \cM_+(X\times X)$, where $m(x,y)$ describes the amount of mass moved from $x\in X$ to $y \in X$. Thus, the optimal transport plan $m$ is the solution to \cite{villani2021topics}
\begin{equation} \label{eq:ot}
    \begin{aligned}
     \minwrt[m \in \cM_+(X\times X)]& \ \int_{X \times X} c(x,y) dm(x,y) \\
        \text{subject to }& \   \int_{A \times X} dm(x,y) = \int_A d\mu(x), \\
        & \  \int_{X \times B} dm(x,y) = \int_B d\nu(x), \\
        & \text{for all measurable sets } A,B \subset X.
    \end{aligned}
\end{equation}
Note that the constraints in \eqref{eq:ot} impose that the transport plan $m$ indeed transports the mass from $\mu$ to $\nu$.

\subsection{Motivating example} \label{subsec:example}

In this work, we consider the setting of separating the transport between two measures into several different transport plans.
As an example, consider two Gaussian mixtures in $ X=\RR$, given by 
\begin{equation} \label{eq:gmms}
\begin{aligned}
    & \mu \sim p \cdot \cN(\gaussmean,\sigma) + p' \cdot  \cN(\gaussmean',\sigma'), \\
    & \nu \sim p' \cdot \cN(\gaussmean,\sigma') + p \cdot  \cN(\gaussmean',\sigma),
\end{aligned}
\end{equation}
with weights $p,p' \in \RR_+$, means $\gaussmean,\gaussmean' \in \RR$, and standard deviations $\sigma,\sigma' \in \RR_+$. 
Note that the modes of the two Gaussian mixtures are identical, and only their locations are switched.

Our framework assumes that the two modes correspond to different subpopulations that are each governed by their own dynamics.
For instance, we could be observing the superposition of two separate phenomena.
Thus, we would like to find two transport plans that move the mass between the corresponding modes of the Gaussian, without splitting up the subpopulations.
Since the modes have exactly the same shape, it makes sense to look for transport plans that preserve the shape of the modes throughout.

The standard optimal transport problem does not take this structural information into account.
This is illustrated in Figure \ref{fig:ex_dist}, which shows the solution to an optimal transport problem \eqref{eq:ot} with cost function $c(x,y)= \|x-y\|_2^2$.
The optimal transport plan assigns a fraction of the mass to move very little. 
Additionally, since the modes do not have the same mass, a large part of mass must be moved from one mode to the other.

In this work we propose to separate the populations by finding several transport plans, which are each associated with a parameterized cost function $c_\theta:X \times X \to \RR_+$, where $\theta$ is a parameter that we optimize over.
Thus, for the example with two Gaussian mixtures in Figure \ref{fig:ex_dist}, our aim is to find two transport plans $m_1,m_2 \in \cM(X\times X)$, and utilize cost functions with a shift parameter $\theta$, that is $$c_\theta(x,y) = \|x-y+\theta\|_2^2.$$
%
This cost allows shifting the mass in each mode jointly, while enforcing that the shape of each mode remains as similar as possible.
\begin{figure}
    \begin{minipage}{0.49\linewidth}
    \includegraphics[trim={50 0 30 0}, clip, width=\linewidth]{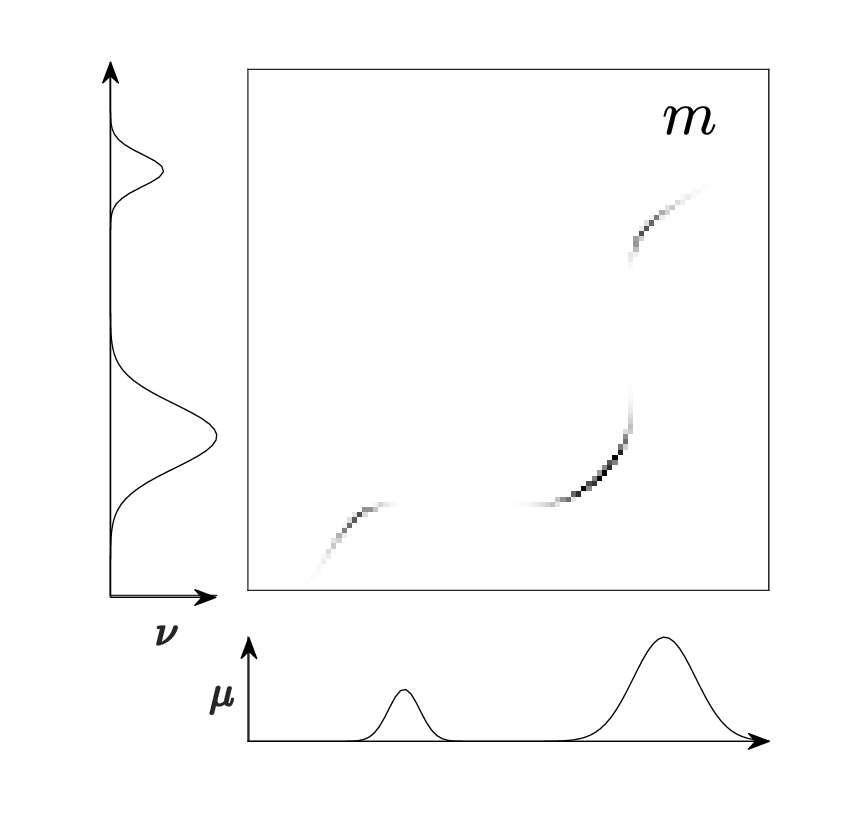} \end{minipage} \hfill
    \begin{minipage}{0.49\linewidth}
   \includegraphics[trim={70 0 55 50}, clip, width=\linewidth]{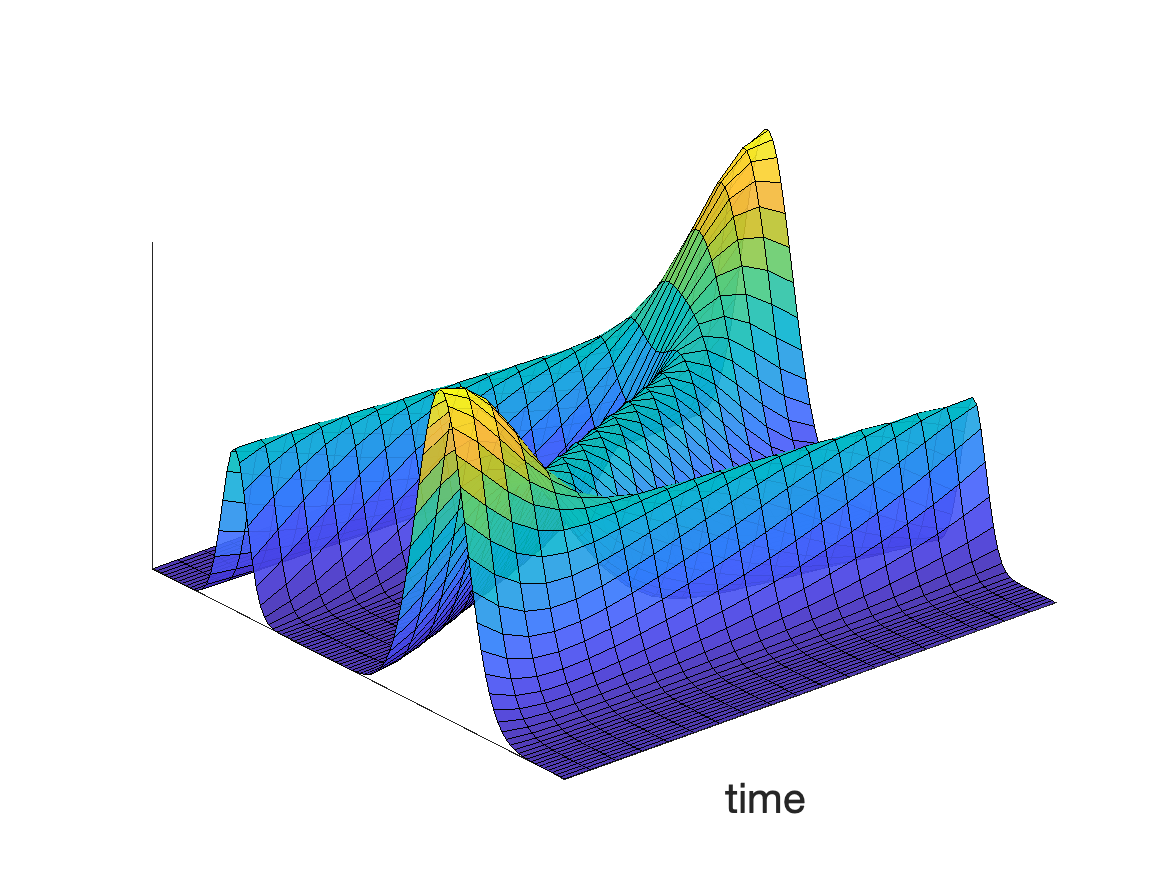}
   \end{minipage}
    \caption{Classical optimal transport between two Gaussian mixtures. The left plot shows the given distributions and optimal the transport plan $m$, where dark areas correspond to support of $m$ in $\RR \times \RR$. The right plot shows the evolution of the distribution over time from $\mu$ to $\nu$.}   
    \label{fig:ex_dist}
\end{figure}
That is, we propose to solve
\begin{align} \label{eq:ot_ex}
     & \minwrt[\substack{m_1,m_2 \in \cM_+(X\times X)\\ \theta_1. \theta_2 \in \RR } ] \ \langle c_{\theta_1}, m_1 \rangle  + \langle c_{\theta_2} , m_2 \rangle \\
      & \text{subject to  }    \int_{A \times X} \!\!\! dm_1(x,y) + \int_{A \times X} \!\!\! dm_2(x,y) = \int_A \!\! d\mu(x), \nonumber \\
        &  \hspace{42pt} \int_{X \times B}  \!\!\!dm_1(x,y) + \int_{X \times B} \!\!\! dm_2(x,y) = \int_B \!\! d\nu(y), \nonumber \\
        & \hspace{42pt} \text{for all measurable sets } A,B \subset X, \nonumber 
\end{align}
where the inner product is defined as
\begin{equation} \label{eq:innerprod}
    \langle c, m \rangle :=  \int_{X \times X} c(x,y) dm(x,y). 
\end{equation}
Given the Gaussian mixtures in \eqref{eq:gmms}, the optimization problem has a unique minimizer, which corresponds to moving the modes of the mixtures in $\mu$ to the corresponding modes in $\nu$.
More precisely, let the modes of $\mu$ be defined as $\mu_1 \sim p \cdot \cN(\gaussmean,\sigma)$ and $ \mu_2 \sim  p' \cdot  \cN(\gaussmean',\sigma')$.
Then, the unique optimal solution of \eqref{eq:ot_ex} is given by the shift parameters and transport plans
\begin{equation*}
    \begin{aligned}
&\theta_1=\gaussmean'-\gaussmean, \quad m_1 = (id, T_1)_\#\mu_1 \\
&\theta_2=\gaussmean-\gaussmean', \quad m_2 = (id, T_2)_\#\mu_2,
    \end{aligned}
\end{equation*}
where $id:\RR \to \RR$ is the identity mapping, the mappings\footnote{These mappings are the Monge optimal transport maps.} 
$T_1,T_2:\RR \to \RR$ are defined as $T_1(x) = x + \theta_1$ and $T_2(x) = x + \theta_2$, and $\#$ denotes the push-forward operator.
In fact, it can be easily checked that with these variables, the minimal objective value in \eqref{eq:ot_ex} is zero. 

In Figure \ref{fig:ex_plans} the solution is demonstrated on the same setting as in Figure \ref{fig:ex_dist}.
\begin{figure}
    \begin{minipage}{0.49\linewidth}
    \includegraphics[trim={50 0 30 0}, clip, width=\linewidth]{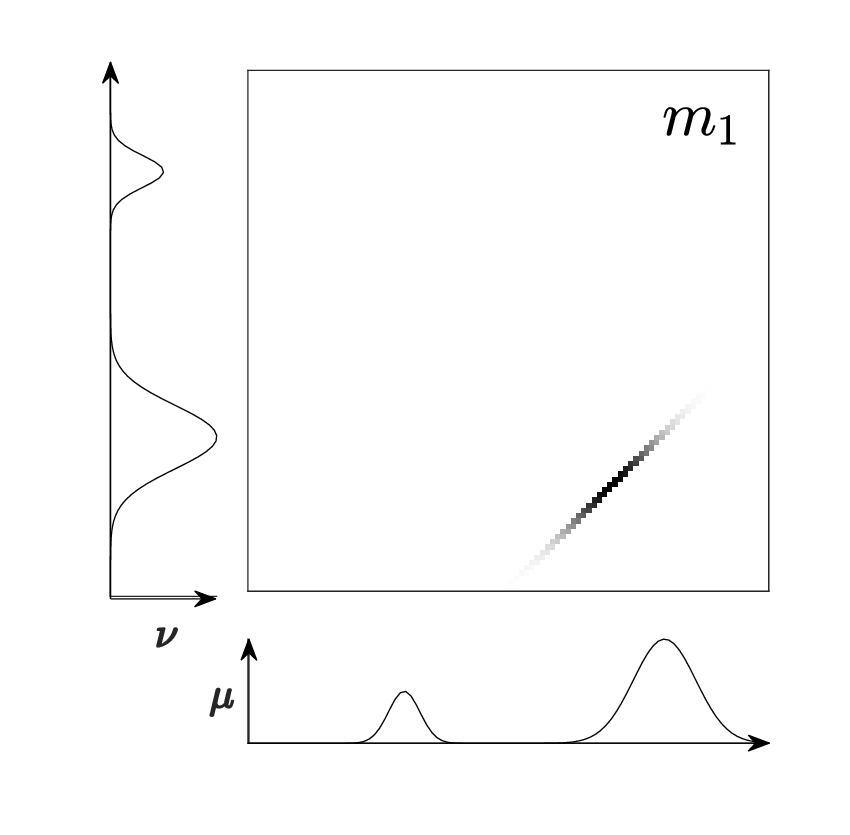}\end{minipage} \hfill
    \begin{minipage}{0.49\linewidth}
   \includegraphics[trim={70 0 55 50}, clip, width=\linewidth]{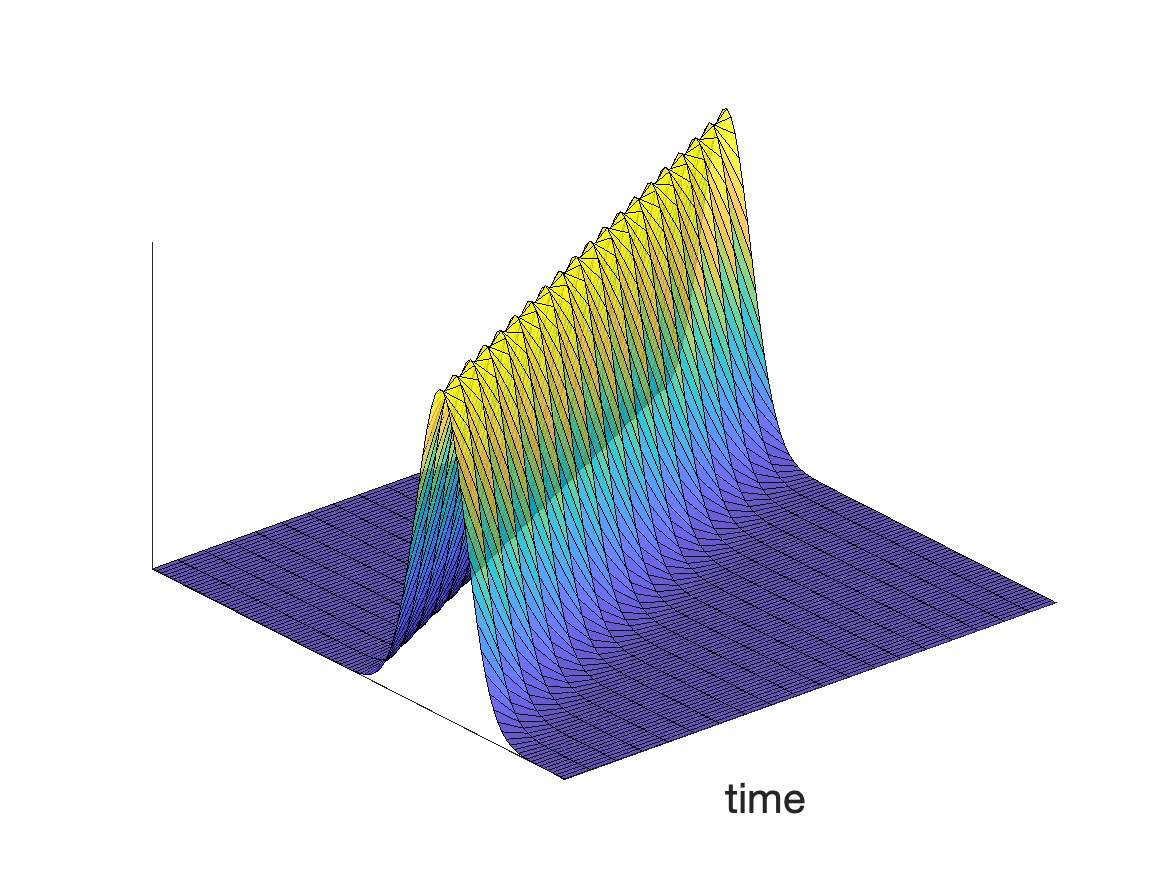} 
   \end{minipage}
   \begin{minipage}{0.49\linewidth}
    \includegraphics[trim={50 0 30 0}, clip, width=\linewidth]{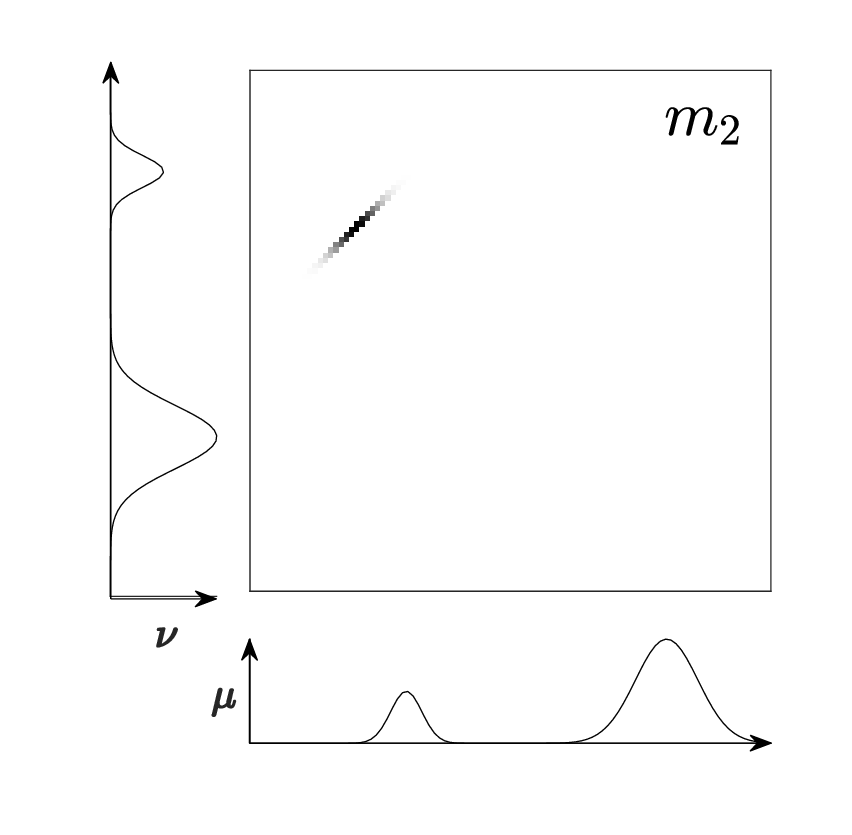}  \end{minipage} \hfill
    \begin{minipage}{0.49\linewidth}
   \includegraphics[trim={70 0 55 50}, clip, width=\linewidth]{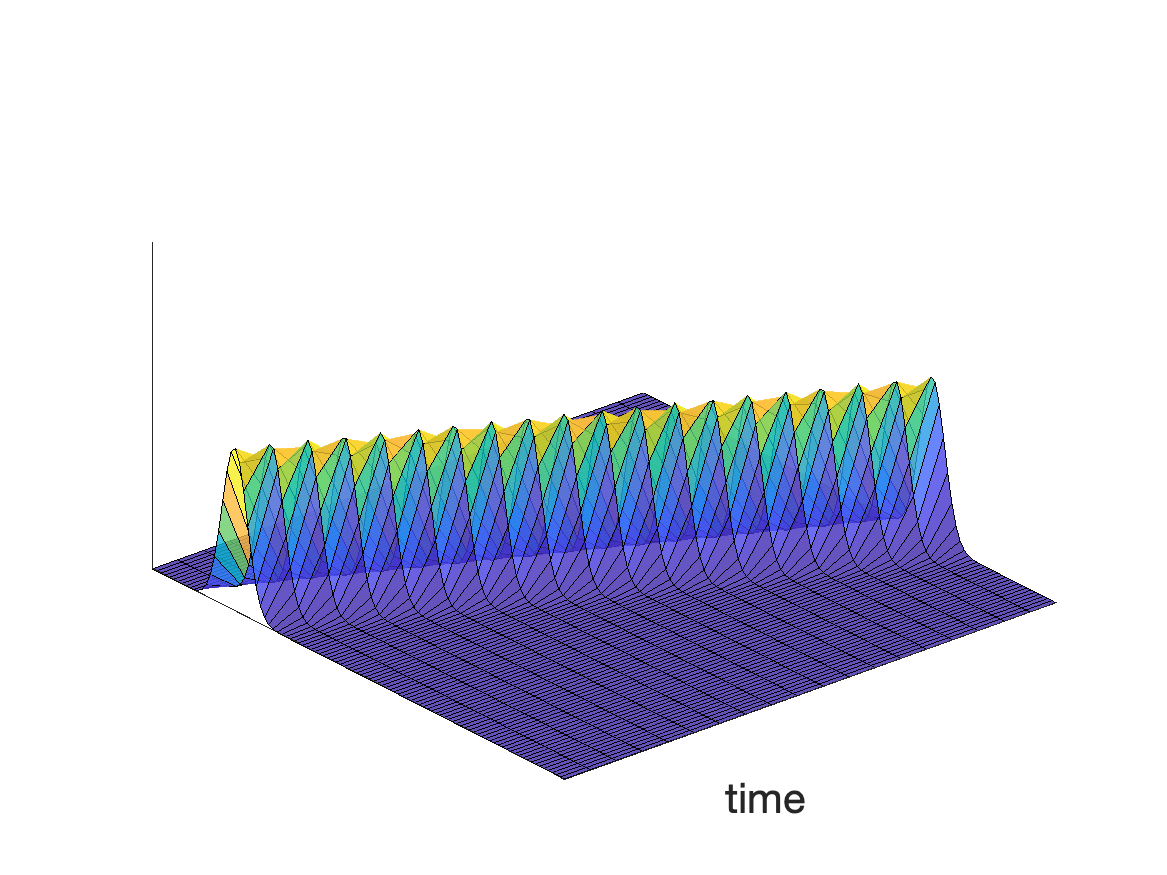} 
   \end{minipage}
    \caption{Separated optimal transport between two Gaussian mixtures. The plots on the left show the given distributions and optimal the transport plans $m_1$ and $m_2$,
where dark areas correspond to support in $\RR^2$. The plots on the right
show the corresponding evolutions of the distributions over time.}
    \label{fig:ex_plans}
\end{figure}
One can see that each transport map corresponds to one of the modes in the Gaussian mixtures.
Moreover, the support of the transport plans $m_1$ and $m_2$ concentrates on the lines ${(x,x+\theta_1), x \in \RR}$ and ${(x,x+\theta_2), x \in \RR}$, respectively.
Thus, the optimal transport cost in \eqref{eq:ot_ex} is zero.

Following the same arguments, any population represented by a Gaussian mixture model with $K$ modes can be separated into $K$ ensembles, when the two given distributions contain the same modes with different means.

\section{Problem formulation} \label{sec:form}

Building on the intuition from the example in Section \ref{subsec:example}, we now propose our general framework to
jointly separate a population into a set of ensembles and
identify the dynamical system governing each ensemble, based on aggregate state observations.

Therefore, consider an ensemble of $K$ populations with state space in $X=\RR^d$.
We assume that each ensemble is described by a discrete-time dynamical system
\begin{equation} \label{eq:dyanmics}
    x^{(t+1)} = \Phi_{\theta_k}(x^{(t)}), \quad \text{ for } k=1,\dots,K,
\end{equation}
where $x^{(t)} \in \RR^d$ is the state of a particle at time $t$, and $\Phi_\theta: \RR^d \to \RR^d$ is a state transition function parameterized by an unknown parameter $\theta \in \RR^P$.
Moreover, we are given aggregate state observations of all ensembles at time instances $t=1,\dots,T$, which are described by measures $\mu^{(t)}\in \cM_+(\RR^d)$, for $t=1,\dots,T$.
These measures may describe a continuous distribution of particles, as in \eqref{eq:gmms},
or they may describe empirical distributions of a finite set of agents, or particles,
\begin{equation} \label{eq:particle_distribution}
    \mu^{(t)} \sim \sum_{k=1}^K \sum_{n=1}^{N_k} \delta(x_{k,n}^{(t)}),
\end{equation}
where $x_{k,n}^{(t)} \in \RR^d$ denotes the location of the $n$-th agent in the $k$-th population at time $t$, and $N_k$ denotes the number of agents in the $k$-th population, as in Figure \ref{fig:ex_3groups_illustration}.
Each observed measure $\mu^{(t)}$ is the superposition of $K$ measures $\mu_k^{(t)} \in \cM_+(\RR^d)$, which describe the distribution of the $k$-th ensemble at time $t$, that is, $\sum_{k=1}^K \mu_k^{(t)} = \mu^{(t)}$ for $t=1,\dots,T$.
We aim to find these distributions $\mu_k^{(t)}$ for $k=1,\dots,K$ and $t=1,\dots,T$,
and to identify the dynamics \eqref{eq:dyanmics} of each ensemble by finding the parameters $\theta_k$, for $k=1,\dots,K$.

\subsection{Optimization problem}

We propose to separate the ensembles and identify the systems by solving an optimal transport problem similar to \eqref{eq:ot_ex}.
For a pair of consecutive states $(x^{(t)},x^{(t+1)})$ in the same ensemble, we define their transport cost as
\begin{align*}
    c_\theta(x^{(t)}, x^{(t+1)}) = \norm{ \Phi_\theta(x^{(t)}) - x^{(t+1)}  }_2^2.
\end{align*}
In order to separate the populations, and identify their dynamics, we find a transport plan $m^{(t)}_k$ for each ensemble $k=1,\dots,K$ and each time step $t=1,\dots,T-1$,
by solving the optimal transport problem
\begin{align} \label{eq:problem_form}
     & \minwrt[\substack{m_k^{(t)} \in \cM_+(X\times X) \\ k=1,\dots,K, t=1,\dots,T-1 \\ \theta_k \in \RR^P,\ \mu_k^{(t)} \in \cM_+(X) \\ k=1,\dots,K, t=1,\dots,T } ] \ \sum_{t=1}^{T-1} \sum_{k=1}^K \langle c_{\theta_k}, m^{(t)}_k \rangle   \\
      & \text{subject to  }   \int_{A \times X} \!\!\! dm^{(t)}_k(x,y) = \int_A  d\mu_k^{(t)}(x), \nonumber \\
        &  \hspace{42pt}  \int_{X \times B}  \!\!\!dm^{(t)}_k(x,y)  = \int_B  d\mu_k^{(t+1)}(x), \nonumber \\
        & \hspace{30pt} \text{for all measurable sets } A,B \subset X, \ t=1,\dots,T-1  \nonumber \\
        & \hspace{42pt}  \sum_{k=1}^K \mu_k^{(t)} = \mu^{(t)}, \quad t=1,\dots,T,  \nonumber
\end{align}
where the inner products in the objective are defined as in \eqref{eq:innerprod}.
Although the problem can be formulated for generic dynamical systems, in this first study we focus on linear dynamics.
\begin{example}[Linear dynamical systems] \label{ex:linear}
Consider the case, where the state transition functions $\Phi_{\theta}$ in \eqref{eq:dyanmics} define linear dynamics
\begin{align}\label{eq:linear_dynamics}
    x^{(t+1)} = A x^{(t)} + b.
\end{align}
%
%
where $A \in \RR^{d\times d}$ and $b \in \RR^d$.
In this case the system is parameterized by $\theta \in \RR^{d(d+1)}$, describing the elements in $ (A,b)$.
The ground transportation cost from $x^{(t)} \in \RR^d$ to $x^{(t+1)} \in \RR^d$ is defined as
\begin{align} \label{eq:cost_linear}
    c_\theta(x^{(t)},x^{(t+1)}) = \norm{Ax^{(t)}+b-x^{(t+1)}}_2^2.
\end{align}
In particular, the example in Section \ref{subsec:example} is a special case of this system with $d=1$, where $A=1$ and $\theta=b$.
\end{example}

\subsection{Solution method}

We solve problem \eqref{eq:problem_form} using a block coordinate descent scheme, similar to \cite{lamoline2024gene}.
The method is summarized in Algorithm \ref{alg:BCD}.
\begin{algorithm}[t]
    \caption{Block Coordinate Descent}\label{alg:BCD}
    \KwData{Initialize  $\{ \theta_k; k=1,\dots,K \}$}
     \While{not converged}{
\begin{enumerate}
    \item  Minimize \eqref{eq:problem_form} with respect to
    $\{m_k^{(t)}\}_{k=1,t=1}^{K,T-1}$ \\ and $\{\mu_k^{(t)}\}_{k=1,t=1}^{K,T}$,\\
    while keeping 
    $\{ \theta_k\}_{k=1}^K$ fixed
    \item  Minimize \eqref{eq:problem_form} with respect to $\{ \theta_k\}_{k=1}^K$,\\
    while keeping $\{m_k^{(t)}\}_{k=1,t=1}^{K,T-1}$ and\\ $\{\mu_k^{(t)}\}_{k=1,t=1}^{K,T}$
    fixed
\end{enumerate}       
     }
\end{algorithm}
It should be noted that problem \eqref{eq:problem_form} is generally highly non-convex.
However, note that the problem is linear in
$\{m_k^{(t)}\}_{k=1,t=1}^{K,T-1}$ and $\{\mu_k^{(t)}\}_{k=1,t=1}^{K,T}$.
In fact, step 1) in Algorithm \ref{alg:BCD} requires solving a sum of coupled optimal transport problems \eqref{eq:ot}.
Moreover, if the dynamics $\Phi_\theta$ in \eqref{eq:dyanmics} are linear in the parameter $\theta$, then \eqref{eq:problem_form} is convex in
$\{ \theta_k\}_{k=1}^K$.
This proves the following result.
\begin{prop} \label{prop:biconvex}
If the dynamics $\Phi_\theta$ in \eqref{eq:dyanmics} are linear in the parameter $\theta$, then the non-convex problem \eqref{eq:problem_form} is bi-convex in the sets
\vspace{-2pt}
\begin{equation*}
\left\{ \{m_k^{(t)}\}_{t=1}^{T-1}, \{\mu_k^{(t)}\}_{t=1}^{T} \right\}_{k=1}^K
\text{ and }
\{ \theta_k\}_{k=1}^K.    
\end{equation*} 
\end{prop}

\vspace{7pt}
In particular, in the case of linear dynamics, as in Example~\ref{ex:linear}, problem \eqref{eq:problem_form} is bi-convex and step 2) in Algorithm~\ref{alg:BCD} requires solving a linear least-squares problem, cf. \cite{lamoline2024gene}.

From Proposition \ref{prop:biconvex} it follows that Algorithm \ref{alg:BCD} results in a sequence of variables such that their objective value converges \cite[Theorem 4.5]{gorski2007biconvex}.
In the setting that we consider in our numerical experiments in Section \ref{sex:exp}, we can even guarantee almost-sure convergence of the variable sequence itself to a local optimum.

\begin{prop} \label{prop:conv}
    Consider problem \eqref{eq:problem_form} with cost \eqref{eq:cost_linear} and discrete measures $\mu{(t)}$ as in \eqref{eq:particle_distribution}.
    Assume that the support points $$\{x_{k,n}^{(t)} ;\ k=1,\dots,K, n=1,\dots,N_k, t=1,\dots,T \}$$
    of the observed distributions \eqref{eq:particle_distribution}
    are sampled from a continuous probability distribution $\cP$ over $\RR^{d^N}$, where ${N= T \cdot N_1 \cdots N_K}$.
    Then, Algorithm \ref{alg:BCD} converges $\cP$-almost-surely to a local optimum of \eqref{eq:problem_form}.
\end{prop}
\begin{proof}
    We denote one realization of support points as
    $X := \{x_{k,n}^{(t)} ;\ k=1,\dots,K, n=1,\dots,N_k, t=1,\dots,T \}$.
    Since step 1) in Algorithm \ref{alg:BCD} is a linear program, where the constraints and costs are specified by the realization $X \sim \cP$, this step has a unique solution $\cP$-almost-surely.
    Moreover, step 2) in Algorithm \ref{alg:BCD} corresponds to solving a linear least squares problem based on $X$ and thus also has a unique solution $\cP$-almost-surely.
     Thus, by \cite[Theorem 4.9]{gorski2007biconvex}, the algorithm converges to a local optimum $\cP$-almost-surely.
\end{proof}

Optimal transport problems with continuous measures are typically solved by discretizing them and solving a linear program. Using the same approach for problem \eqref{eq:problem_form} with cost \eqref{eq:cost_linear}, Algorithm \ref{alg:BCD} converges to a local optimum of the discretized problem by a small modification of Proposition \ref{prop:conv}. \looseness=-1

Convergence to a global optimum can typically not be guaranteed for bi-convex problems. In order to avoid finding only suboptimal solutions, in practice we initialize Algorithm~\ref{alg:BCD} with several different values for $\theta$.

\section{Numerical experiments} \label{sex:exp}
To illustrate the proposed method, consider distributions on the form \eqref{eq:particle_distribution} where the state $x \in \RR^2$ is evolving according to linear dynamics on the form \eqref{eq:linear_dynamics}. We consider $K = 3$ ensembles consisting of $N_1 = 10$, $N_2 = 12$, and $N_3 = 15$ particles, respectively, observed at $T = 7$ time points. With this, we perform a Monte Carlo simulation where we in each simulation instance randomly sample dynamics $\theta_k = (A_k,b_k) \in \RR^{2\times 2}\times \RR^2$, $k = 1,2,3$, and the initial state for each particle from a Gaussian distribution (each matrix/vector element is independent standard Gaussian). The states then evolve according to the respective dynamics for $t = 2,\ldots,T$, but we for each $t$ also perturb each state with a white Gaussian noise with variance $\sigma^2$. The three first time steps of a sample trajectory are shown in Figure~\ref{fig:ex_3groups_illustration} for $\sigma^2 = 10^{-3}$.
\begin{figure}
    \centering
    \includegraphics[trim={20 0 40 10}, clip, width=\linewidth]{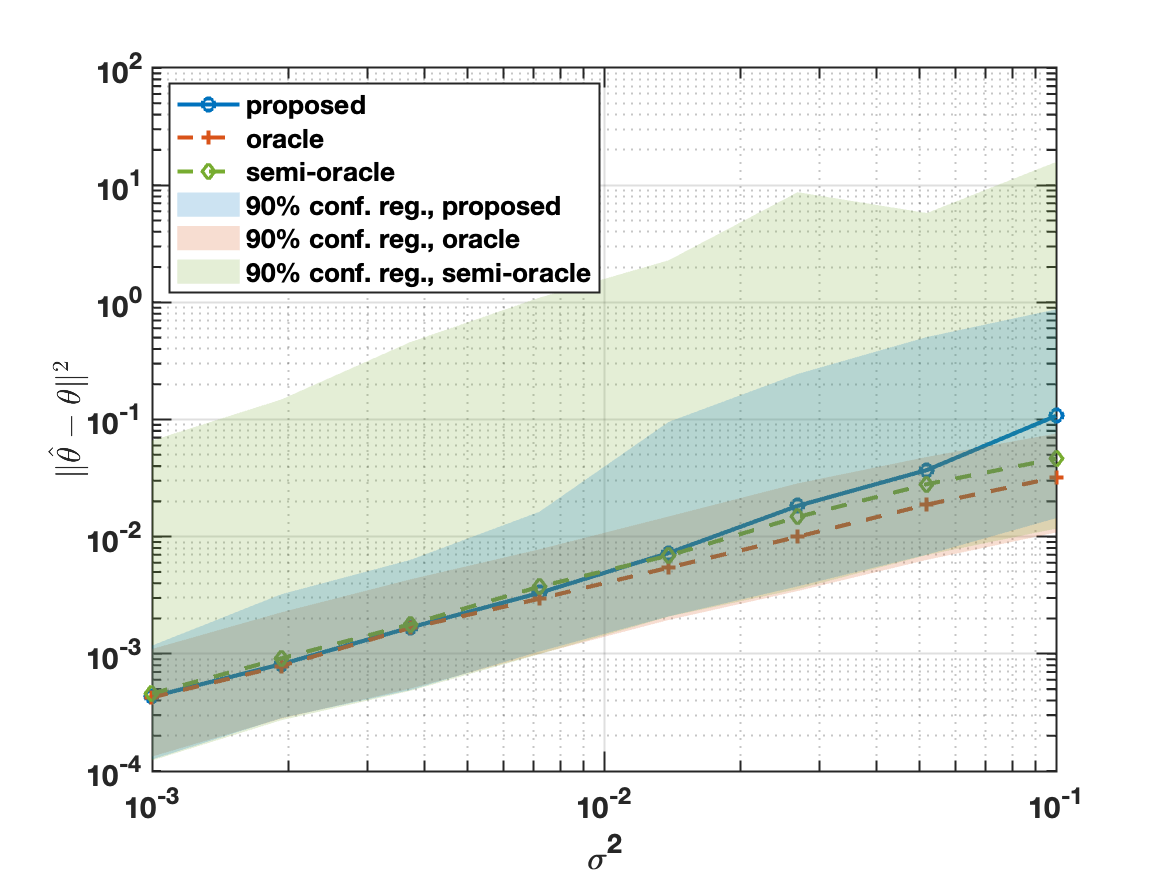}
    \caption{Squared error for estimates of the ensemble dynamics parameters for discrete distributions observed over $T = 7$ time points, for varying noise level. The number of ensembles is $K = 3$. The lines correspond to the median over 500 simulations, and the confidence regions cover 90\% of the observed errors.}
    \label{fig:ex_3groups_param_error}
\end{figure}

With this, we use the proposed problem \eqref{eq:problem_form} to estimate the dynamics parameter $\theta_k$, $k = 1,2,3$, as well as the ensembles. It may be noted that we do not assume knowledge of $N_k$, $k = 1,2,3$. We implement \eqref{eq:problem_form} by the block coordinate descent method in Algorithm~\ref{alg:BCD}. As initial points, we randomly draw candidate dynamics from a Gaussian distribution and run the algorithm until convergence. We do this for 10 initial points and pick as the estimate the limit point with the smallest objective value. The error $\norm{\hat{\theta} - \theta}^2 = \sum_{k=1}^K \norm{\hat{\theta}_k - \theta_k}^2$, where $\hat{\theta} = \{ \hat{\theta}_1, \hat{\theta}_2, \hat{\theta}_3 \}$ is the estimate of the dynamics parameters, is shown in Figure~\ref{fig:ex_3groups_param_error} for different values of the noise variance $\sigma^2$. The presented error is the median of 500 simulations for each value of $\sigma^2$. Also presented are empirical confidence regions covering 90\% of the observed parameter errors.
We have superimposed corresponding quantities for two reference methods; one oracle and one semi-oracle. The oracle estimator is the least-squares estimator of $\theta_1$, $\theta_2$, $\theta_3$ that has full information of the particles' trajectories and their ensemble identities. The semi-oracle has access to the particle trajectories\footnote{That is, for the semi-oracle the particles are not indistinguishable.} but not the ensemble identities. The dynamics parameter are estimated for each particle trajectory individually and then clustered using the K-means algorithm, after which the common ensemble dynamics are re-estimated by least-squares\footnote{It may be noted that the re-estimation step yields higher statistical accuracy than simply using the cluster centroids.}. For the semi-oracle, the K-means clustering is run 100 times, using different random starting points for the cluster centroids.
As can be seen from Figure~\ref{fig:ex_3groups_param_error}, the proposed method performs on par with the oracle for the lower noise levels. For larger state noise, the error increases as a smaller fraction of the particles get correctly grouped in the respective ensembles. The estimated probability of correct classification, defined as the fraction of particles being correctly labeled, is presented in Figure~\ref{fig:ex_3groups_success_probability}. As may be seen, the proposed method maintains high classification accuracy also for higher noise levels. It may be noted that although the semi-oracle has access to full trajectories, the estimates of the corresponding individual dynamics have low accuracy, which causes the clustering to fail often enough as to yield quite large observed errors.
\begin{figure}
    \centering
    \includegraphics[trim={20 0 40 10}, clip, width=\linewidth]{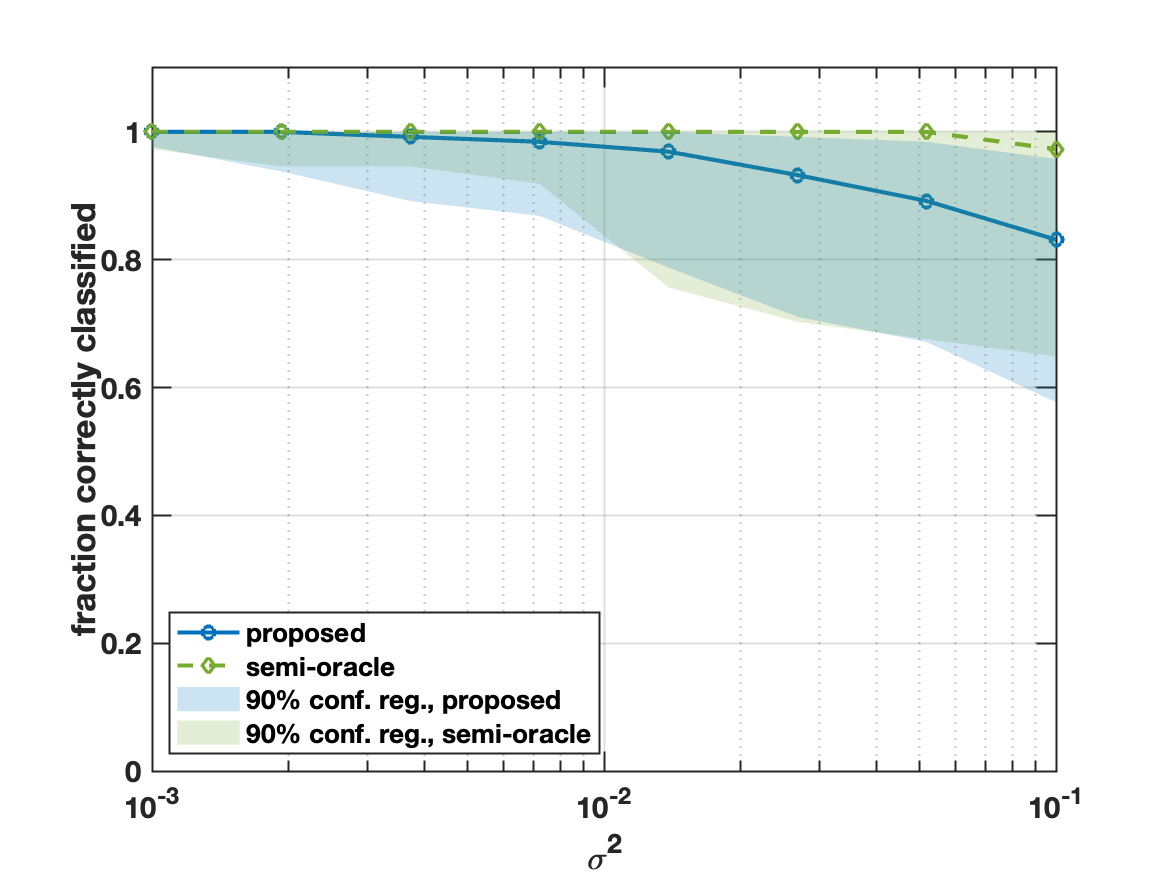}
    \caption{The fraction of particles correctly grouped into their corresponding ensembles for the same scenario as Figure~\ref{fig:ex_3groups_param_error}. The lines correspond to the median over 500 simulations, and the confidence regions cover 90\% of the observed classification fractions.}
    \label{fig:ex_3groups_success_probability}
\end{figure}
\section{Conclusions}

In this work, we introduced an optimal transport framework to identify and separate subpopulations with distinct dynamics from aggregate observations.
We formulated the problem as a bi-convex optimization problem and proposed to solve it using a block coordinate descent method with convergence guarantees.
Numerical experiments confirmed that our method achieves close-to-oracle performance even in noisy settings, highlighting its robustness and practical applicability.

We see a great potential for the proposed framework for problems in control theory, as well as related fields of research such as signal processing.
For example, we have in a follow-up paper applied our framework to the problem of separating audio sources for real speech data \cite{fabiani2025joint}.
Another interesting application domain is in shape matching; we note that our proposed framework has similarities with the well-known iterative closest point method in computer vision \cite{besl1992method}.

\balance

\bibliography{ref}
\bibliographystyle{abbrv}



\end{document}